\documentclass[letterpaper, 10 pt, conference]{ieeeconf}  

\IEEEoverridecommandlockouts                              

\usepackage{amsmath,amssymb,amsfonts}
\usepackage{xcolor}
\usepackage{graphicx}
\allowdisplaybreaks

\usepackage{nomencl}
\makenomenclature

\newtheorem{theorem}{Theorem}
\newtheorem{lemma}{Lemma}
\newtheorem{remark}{Remark}
\newtheorem{corollary}{Corollary}
\newtheorem{assumption}{Assumption}
\newtheorem{definition}{Definition}
\usepackage{hyperref}
\usepackage{cleveref}
\usepackage{cite}

\usepackage{textcomp}
\def\BibTeX{{\rm B\kern-.05em{\sc i\kern-.025em b}\kern-.08em
    T\kern-.1667em\lower.7ex\hbox{E}\kern-.125emX}}
    
\author{Mengmou Li, Masaaki~Nagahara
\thanks{The authors are with the Graduate School of Advanced Science and Engineering, Hiroshima University, Higashi-Hiroshima City, Japan (e-mail: mmli.research@gmail.com; nagahara@ieee.org).}
\thanks{This work is supported by JSPS KAKENHI under grant numbers 24K23864, 23K26130, 22KK0155, 22H00512, and 24K21314.}
}

\title{Exponential Convergence of Augmented Primal-dual Gradient Algorithms for Partially Strongly Convex Functions}
\begin{document}
\maketitle
\begin{abstract}
	We show that the augmented primal-dual gradient algorithms can achieve global exponential convergence with partially strongly convex functions. In particular, the objective function only needs to be strongly convex in the subspace satisfying the equality constraint and can be generally convex elsewhere, provided the global Lipschitz condition for the gradient is satisfied. This condition implies that states outside the equality subspace will converge towards it exponentially fast.
    The analysis is then applied to distributed optimization, where the partially strong convexity can be relaxed to the restricted secant inequality condition, which is not necessarily convex. This work unifies global exponential convergence results for some existing centralized and distributed algorithms.
\end{abstract}

\section{Introduction}
The convergence of optimization algorithms has been an important research focus due to their 
wide range of applications.
In particular, exponential convergence in continuous time, or linear convergence in discrete time, has attracted considerable attention recently.
Researchers have proposed various conditions alternative to strong convexity of the objective function to ensure exponential/linear convergence, such as the Polyak-\L{}ojasiewicz (PL) condition \cite{polyak1963gradient,lojasiewicz1963topological}, the error bound (EB) condition \cite{luo1993error}, the restricted secant inequality (RSI) condition \cite{zhang2015restricted}, and the quadratic growth (QG) condition \cite{anitescu2000degenerate}. Notably, these conditions may not even require convexity \cite{karimi2016linear}.

Convergence analysis becomes more complicated when constraints are introduced.
The augmented Lagrange multiplier method for handling equality constraints was introduced by \cite{hestenes1969multiplier, powell1969method}, and has been extensively studied since then \cite{luo2008convergence}.
Meanwhile, the exponential convergence has been of particular interest in recent years \cite{dhingra2018proximal,qu2019exponential,tang2020semi,chen2020exponential,wang2021exponential}.
Specifically, \cite{dhingra2018proximal} and \cite{qu2019exponential} address this issue for strongly convex functions with linear equality and inequality constraints, respectively. For general nonlinear inequality constraints, only the so-called semi-global exponential stability is obtained without considering constraints to be linearly independent \cite{tang2020semi}. Currently, it remains uncertain whether a global exponential stability condition can be established for augmented primal-dual gradient dynamics in problems involving general nonlinear inequality constraints. The work in \cite{wang2021exponential} extends exponential stability to non-smooth objective functions and general set constraints.
Moreover, its application in distributed optimization is also extensively studied \cite{liang2019exponential,li2020input,yi2020exponential}.

On the other hand, exponential convergence of primal-dual gradient algorithms by harnessing strong convexity from constraints is seldom considered. The most related work is \cite{chen2020exponential}, which addresses partially non-strongly convex problems but with decoupled objective functions and equality constraints. In this work, we aim to demonstrate explicitly how equality constraints can contribute to exponential convergence while relaxing the requirement for strong convexity.
Our theoretical results are shown by utilizing the methods of integral quadratic constraints (IQCs) \cite{megretski1997system} and the Kalman-Yakubovich-Popov (KYP) lemma \cite{rantzer1996kalman}. The proposed framework unifies and extends some existing results in both centralized and distributed optimization. Moreover, a comprehensive understanding of exponential convergence is provided in these contexts.

\text{Notation:} $\mathbb{R}$ and $\mathbb{C}$ represent the sets of real and complex numbers, respectively. $\mathbb{R}^{n}$ and $\mathbb{R}^{n \times m}$ denote the sets of $n$-dimensional real vectors and $n \times m$ real matrices, respectively.
The imaginary unit is denoted by $j$; the real part and imaginary part of a complex number are denoted by $\textup{Re}(\cdot)$ and $\textup{Im}(\cdot)$, respectively.
The Kronecker product is denoted by $\otimes$. The Euclidean norm is denoted by $\| \cdot \|$. Given symmetric matrices $A$ and $B$, $A \prec B$ means that $B- A$ is positive definite, and $\succ$, $\preceq$, $\succeq$ are defined \textit{mutatis mutandis}. The symbols $I_{n}$, $\mathbf{1_n}$, and $\mathbf{0}_{n \times m}$ represent $n\times n$ identity matrix, $n$-dimensional vector with all entries equal to one, and the $n \times m$ zero matrix, respectively, while their subscripts may be omitted when the dimensions are clear from the context.

\section{Preliminaries}
Consider the equality-constrained optimization problem 
\begin{equation}\label{eq:problem}
\begin{aligned}
	&\min_{x \in \mathbb{R}^{n}} f(x) \\
		& \text{subject to } Tx = b
\end{aligned}
\end{equation}
where the objective function $f(x): \mathbb{R}^{n} \rightarrow \mathbb{R}$ is differentiable, $T \in \mathbb{R}^{m \times n}$ with $m \leq n$, and $b\in \mathbb{R}^{m}$. The following assumptions are commonly considered.
\begin{assumption}\label{assumption general convex}
$f(x)$ is convex, i.e., for any $ x,~ x' \in \mathbb{R}^{n}$,
\begin{align*}
    \left( \nabla f(x) - \nabla f(x') \right)^\top (x - x') \geq 0.
\end{align*}
\end{assumption}
\begin{assumption}\label{assumption differentiable lips}
	$f(x)$ is $l$-Lipschitz smooth, i.e., there exists some constant $l > 0$, such that for any $x,~x' \in \mathbb{R}^{n}$,\begin{align*}
  \left( \nabla f(x) - \nabla f(x') \right)^\top (x - x') \leq l \| x - x' \|^2.
\end{align*}
\end{assumption}
\begin{assumption}\label{assumption constraint qualification}
	The matrix $T$ is of full row rank and satisfies $ \kappa_1 I_m \preceq T T^\top  \preceq \kappa_2 I_m$ for some $\kappa_2 \geq \kappa_1 > 0$.
\end{assumption}
Assumption~\ref{assumption constraint qualification} is known as the linear independence constraint qualification (LICQ) \cite{nocedal1999numerical}.
It has been shown in \cite{qu2019exponential} that a primal-dual gradient algorithm for \eqref{eq:problem} achieves global exponential convergence if $f(x)$ is strongly convex.
We aim to show that the objective function can be relaxed to partially strong convexity while ensuring global exponential convergence for augmented primal-dual algorithms.

In the following, we provide the definitions of partitioned vector and partially strong convexity, which will be used in this work.
\begin{definition}
Let $S$ be a subset of $\{1,2,\ldots,n\}$ and assume the number of entries in $S$ is $m\leq n$.
For $x\in\mathbb{R}^n$, the restriction of $x$ to $S$ is denoted by $x_{S}\in\mathbb{R}^m$,
which we call the partitioned vector of $x$ with $S$.
\end{definition}

\begin{definition}
The function $f(x)$ is called $\mu$-partially strongly convex if it is $\mu$-strongly convex with respect to the partitioned vector $x_S$ with some index set $S\subset\{1,2,\ldots,n\}$. If $f(x)$ is differentiable, the $\mu$-partially strong convexity is represented by
\begin{align*}
        \left( \nabla f(x) - \nabla f(x') \right)^\top (x - x') \geq \mu \|x_{S} - x'_{S} \|^2
\end{align*}
for any $x,~x' \in \mathbb{R}^{n}$, where $x_{S}$ and $x'_{S}$ are the partitioned vector of $x$ and $x'$ with $S$, respectively.
\end{definition}

\begin{remark}
A global $\mu$-strongly convex function is also partially strongly convex since
\begin{align*}
\left( \nabla f(x) - \nabla f(x') \right)^\top (x - x') \geq \mu \|x - x'\|^2
 \geq \mu \|x_{S} - x'_{S} \|^2.
 \end{align*}
\end{remark}
Let $Q \in \mathbb{R}^{n \times n}$, $R \in \mathbb{R}^{n \times m}$ be a QR decomposition of $T^\top $, i.e., $Q$ is an orthogonal matrix and $R$ is an upper triangular matrix such that
\begin{align}
	T^\top  = QR =
	\begin{bmatrix}
		Q_1 & Q_2
	\end{bmatrix}
	\begin{bmatrix}
		R_1 \\ \mathbf{0}
	\end{bmatrix}
	= Q_1 R_1
\end{align}
where $R_1 \in \mathbb{R}^{m \times m}$, $Q_1 \in \mathbb{R}^{n \times m}$ and $Q_2 \in \mathbb{R}^{n \times (n-m)}$. 
We relax the strong convexity requirement of $f(x)$ over the whole domain in the problem \eqref{eq:problem} to being strongly convex within the subspace satisfying the equality constraints. Let us define state $x'$ by the linear transformation $x = Qx'$ and consider the following assumption.
\begin{assumption}\label{assumption partially strongly convex}
When $m < n$, $g(x') := f(Qx')$ is $\mu$-partially strongly convex with respect to the partitioned vector $x'_{S}$ with $S=\{n-m+1,n-m+2,\ldots,n\}$,
namely, the vector $x'$ is partitioned as $x'=\left(x'_{S^c}, x'_{S} \right)$ where $S^c=\{1,2,\ldots,n-m\}$.
\end{assumption}
\begin{remark}
It can be easily shown that convexity is preserved under invertible linear transformation. 
Assumption~\ref{assumption partially strongly convex} is additionally imposed since the solution to $T x = T Q x' = \begin{bmatrix}
R_1^\top & \mathbf{0}
\end{bmatrix}
x' = b$ is irrelevant to the value of 
$x'_{S}$,
and thus $T$ has no effect on 
$x'_{S}$.
\end{remark}

Finally, we assume the existence of a solution to problem \eqref{eq:problem}:
\begin{assumption}\label{assumption feasible problem}
Problem \eqref{eq:problem} has a finite and feasible solution.
\end{assumption}

\section{Augmented Primal-dual algorithm}
The augmented Lagrangian for problem \eqref{eq:problem} is
\begin{align}\label{eq:augmented Lagrangian}
\mathcal{L} (x, \lambda)= f(x) + \lambda^\top (T x - b) + \frac{\alpha}{2} \|Tx - b\|^2
\end{align}
where $\lambda \in \mathbb{R}^{m}$ is the multiplier and $\alpha > 0$ is a penalty parameter. 
With this, the augmented primal-dual gradient algorithm for \eqref{eq:problem} is given by
\begin{equation}\label{eq:primal-dual algorithm}
\begin{aligned}
\dot{x} =& - \nabla f(x) - T^\top  \lambda - \alpha T^\top  (Tx - b)\\
	\dot{\lambda} =& T x - b.
\end{aligned}
\end{equation}
It is well-known that any equilibrium to the above dynamics satisfies the Karush–Kuhn–Tucker (KKT) conditions \cite{boyd2004convex} and is thus an optimal solution to the problem under Assumptions~\ref{assumption general convex}, \ref{assumption differentiable lips}, and \ref{assumption constraint qualification}.
In addition, exponential convergence of \eqref{eq:primal-dual algorithm} is guaranteed when $f(x)$ is strongly convex.
Nevertheless, we observe that when $f(x) = 0$, the error system of \eqref{eq:primal-dual algorithm} with respect to an equilibrium $(x^{*}, \lambda^{*})$ becomes
\begin{align}\label{eq:error system}
	\begin{bmatrix}
		\dot{\tilde{x}} \\ \dot{\tilde{\lambda}}
	\end{bmatrix}
	 = \begin{bmatrix}
	 	-\alpha T^\top  T & -T^\top \\
	 	T & \mathbf{0}
	 \end{bmatrix}
	 \begin{bmatrix}
		\tilde{x} \\ \tilde{\lambda}
	\end{bmatrix}
\end{align}
where $\tilde{x} = x - x^{*}$ and $\tilde{\lambda} = \lambda - \lambda^{*}$.
It can be shown that the system matrix in \eqref{eq:error system} is Hurwitz if $m = n$.

Let us first introduce the following lemma that shows the Hurwitz property for matrices in a specific form.
\begin{lemma}\label{lem:Hurwitz}
Consider the real matrix
\begin{align*}
A = \begin{bmatrix}
	 	- F & -T^\top \\
	 	T & \mathbf{0}
\end{bmatrix}
\end{align*}
where $F \in \mathbb{R}^{n \times n}$, $T \in \mathbb{R}^{m \times n}$, and $m \leq n$.
If $F \succ 0$, and $T$ is of full row rank, then $A$ is Hurwitz.
\end{lemma}
\begin{proof}
Since $A + A^\top  \preceq 0$, all the eigenvalues $s$ of $A$ have non-positive real parts. Let the non-zero vector $z = (x,  \lambda)$ be an eigenvector of $A$ such that 
\begin{align}\label{eq:eigenvector}
	A  z = s  z
\end{align}
and $\overline{z} = \left(\overline{x},  \overline{\lambda}\right)$ be the complex conjugate to $ z$. Then we have
\begin{align*}
	\text{Re}\left(s \| z\|^2 \right) = & \text{Re} \left({\overline{z}}^\top  A z \right) \\
	=& \text{Re} \left( - \overline{x}^\top  F x - \overline{x}^\top  T^\top \lambda + \overline{\lambda}^\top  T x\right)\\
	= & \text{Re} \left( - \overline{x}^\top  F x - j \text{Im}\left( \overline{x}^\top  T^\top  \lambda \right) \right)\\
	= & \text{Re} \left( - \overline{x}^\top  F x \right).
\end{align*}
Assume that $\text{Re}(s) = 0$, then $ x = 0$ since $F \succ 0$. Substituting it into \eqref{eq:eigenvector}, we have
	$- T^\top \lambda = 0$. 
Recall that $T$ has full row rank, then $T^\top \lambda = 0$ implies that $\lambda = 0$, which contradicts that the eigenvector $ z = (x,  \lambda) \neq 0$. Therefore, $\text{Re}(s) < 0$, and thus $A$ is Hurwitz.
\end{proof}
This result shows that when $m = n$, exponential convergence for system \eqref{eq:error system} can be obtained by letting $F = T^{\top} T$ in Lemma~\ref{lem:Hurwitz}, without any requirement on the strong convexity. This implies that we only need partially strong convexity as stated in Assumption~\ref{assumption partially strongly convex} to enforce exponential convergence of the states ``irrelevant'' to the equality subspace when $m \leq n$.

\begin{remark}
We consider the augmented Lagrangian in this work because the standard one fails to guarantee exponential convergence without the penalty term. The standard primal-dual gradient algorithm to solve problem \eqref{eq:problem} is given by
\begin{align*}
	\dot{x} =& - \nabla f(x) - T^\top  \lambda \\
	\dot{\lambda} =& Tx - b.
\end{align*}
The objective function becomes irrelevant to the optimal solution when $m = n$. However, when it is affine linear, i.e., $f(x) = Fx + c$, the error system becomes 
\begin{align}\label{eq:classic error system}
	\begin{bmatrix}
		\dot{\tilde{x}} \\ \dot{\tilde{\lambda}}
	\end{bmatrix}
	 = \begin{bmatrix}
	 	\mathbf{0} &  -T^\top \\ T &\mathbf{0}
	 \end{bmatrix}
	 \begin{bmatrix}
		\tilde{x} \\ \tilde{\lambda}
	\end{bmatrix}.
\end{align}
The system trajectories of \eqref{eq:classic error system} keep oscillating without convergence \cite{yamashita2020passivity,li2022parallel}.
\end{remark}

\section{Exponential convergence by IQC}
\begin{theorem}\label{thm exponential convergence}
	Under Assumptions~\ref{assumption general convex}, \ref{assumption differentiable lips}, \ref{assumption constraint qualification}, \ref{assumption partially strongly convex} and \ref{assumption feasible problem}, the variable $x(t)$ in the augmented primal-dual gradient algorithm \eqref{eq:primal-dual algorithm} exponentially converges to the optimal solution of problem \eqref{eq:problem}.
\end{theorem}
\begin{proof}
We utilize the method of IQCs \cite{megretski1997system} and the KYP lemma \cite{rantzer1996kalman} to prove exponential convergence of algorithm \eqref{eq:primal-dual algorithm}. 
The proof is divided into two steps.\\
\textbf{Step 1}: when $m = n$:\\
The error system dynamics can be written as
\begin{align}\label{eq:error system general}
& \dot{z} = A z + B u, \quad y = C z, \quad u = \Delta(y)
\end{align}
where $z = \left( \tilde{x}, \tilde{\lambda} \right)$,
\begin{align*}
A = \begin{bmatrix}
	 	 -T^\top  T & -T^\top \\
	 	 T & \mathbf{0}
	 \end{bmatrix},
	 B = \begin{bmatrix}
	 	-I \\ \mathbf{0}
	 \end{bmatrix}, 
	 C = \begin{bmatrix}
	 	I & \mathbf{0}
	 \end{bmatrix}
\end{align*}	 
and $u$ is nonlinearity $\Delta (y)$ added to the systems, given by
\begin{align*}
	u = \Delta(y) := \nabla f(y + x^*) - \nabla f(x^*),
\end{align*}
or equivalently, $u = \nabla f(x) - \nabla (x^*)$, and $y = x - x^*$. Since $f(x)$ is convex and $\nabla f(x)$ is $l$-Lipschitz, it satisfies the co-coercivity\cite{lessard2016analysis}, i.e., 
\begin{align*}
\left( \nabla f(x) - \nabla f(x^*) \right)^\top  (x - x^*)
	\geq \frac{1}{l} \| \nabla f(x) - \nabla f(x^*) \|^2.
\end{align*}
We can obtain the following IQC,
\begin{align}\label{eq:IQC}
	\int_{0}^{T}
	\begin{bmatrix}
		y(t) \\ u (t)
	\end{bmatrix}^\top 
	\Pi
	\begin{bmatrix}
		y (t) \\ u(t)
	\end{bmatrix}
	d t \geq 0, ~ 
\Pi = \begin{bmatrix}
		0 & l \\
	  l  & -2
	\end{bmatrix} \otimes I_{n}
\end{align}
for all $T \geq 0$.
According to \cite{hu2016exponential}, the closed-loop system \eqref{eq:error system general} with the nonlinearity $\Delta(y)$ that satisfies \eqref{eq:IQC} is $\rho$-exponential convergent, i.e., $\|z (t) - z^*\| \leq c e^{-\rho t} \|z(0) - z^*||$, for some $c \geq 0$, and $\rho > 0$ if there exists a $P \succ 0$ such that,
\begin{align}\label{eq:iqc LMI}
	\begin{bmatrix}
	A_{\rho}^\top  P + P A_{\rho} & PB\\
	B^\top  P & \mathbf{0}
\end{bmatrix}
+ \begin{bmatrix}
	C &\mathbf{0}\\ \mathbf{0} & I
\end{bmatrix}^\top 
\Pi
\begin{bmatrix}
	C & \mathbf{0} \\ \mathbf{0} & I
\end{bmatrix} \preceq 0
\end{align}
where $A_{\rho} = A + \rho I$.
By the KYP lemma \cite{rantzer1996kalman}, an equivalent frequency domain inequality for \eqref{eq:iqc LMI} can be obtained by
\begin{align}\label{eq:KYP}
\begin{bmatrix}
	G_{\rho}(- j\omega) \\ I
\end{bmatrix}^{\top}
\Pi
\begin{bmatrix}
	G_{\rho}(j\omega) \\ I
\end{bmatrix}
\preceq 0, ~ \text{for all } \omega \in \mathbb{R}
\end{align}
where $G_{\rho}(j\omega) = C (j\omega I - A_{\rho})^{-1} B$.
Substituting $\Pi$ of \eqref{eq:IQC} into \eqref{eq:KYP}, we obtain the condition 
\begin{align}\label{eq:IFP}
- G_{\rho}(j\omega) - G_{\rho}^\top ( - j\omega)
 \succeq - \frac{2}{l} I, ~ \text{for all } \omega \in \mathbb{R}
\end{align}
which is equivalent to 
system $(A_{\rho},B,-C)$ being input feedforward passive (IFP) with index $- \frac{2}{l}$ \cite{Khalil2002,li2022parallel}.
Select the storage function 
\begin{align*}
V_s = \frac{1}{2} \|x - x^* \|^2 + \frac{1}{2} \|\lambda - \lambda^* \|^2 \geq 0, 
\end{align*}
the derivative along the system $(A,B,-C)$ gives
\begin{align*}
\dot{V}_s = - x^\top  T^\top  T x + y^\top  u,
\end{align*}
meaning that the system $(A,B,-C)$ is passive, and $- G(j\omega) - G^\top ( - j\omega) \succeq 0$, for all $\omega \in \mathbb{R}$ \cite{Khalil2002}.


Since $A$ is Hurwitz by Lemma~\ref{lem:Hurwitz}, $G(s)$ has no poles for $\textup{Re} (s) > - \varepsilon$, i.e., it is analytic in the open vertical strip
\begin{align*}
\left\{ s \in \mathbb{C}  \mid \textup{Re} (s) > - \varepsilon \right\}
\end{align*}
for some $\varepsilon > 0$. Then $G (- s)$ is analytic for $\textup{Re} (-s) > - \varepsilon$. This implies that $G^{\top} (-s)$ is analytic for $\textup{Re} (s) < \varepsilon$. Therefore, the sum $G(s) + G^{\top} (-s)$ is analytic in the open vertical strip
\begin{align*}
\left\{ s \in \mathbb{C} \mid - \varepsilon < \textup{Re} (s)  < \varepsilon \right\}.
\end{align*}
Recall that $G_{\rho}(j \omega ) = G( j \omega - \rho)$, there exists a sufficiently small $\rho > 0$ such that $G_{\rho}(j \omega ) + G_{\rho}^{\top} (- j\omega)$ is analytic for all $\omega \in \mathbb{R}$.
As it remains analytic as $\rho$ decreases to zero, and
\begin{align*}
-\lim_{\rho \rightarrow 0} 
\left( G_{\rho}( j \omega) + G_{\rho}^\top ( - j \omega) \right) = - G(j\omega) - G^\top ( - j\omega) \succeq 0, 
\end{align*}
for all $\omega \in \mathbb{R}$, we conclude that there exist sufficiently small $\delta \geq 0$ and $\rho (\delta) > 0$ such that \eqref{eq:IFP} holds, i.e.,
\begin{align*}
- G_{\rho(\delta)}( j \omega) - G_{\rho(\delta)}^\top ( - j \omega)  \succeq  - \delta I  \succeq - \frac{2}{l} I, ~\text{for all }\omega \in \mathbb{R}.
\end{align*}
Then, the system is $\rho(\delta)$-exponentially stable.\\
%
%
%
%
\textbf{Step 2}: when $m < n$:\\
The matrix $A$ has $(n-m)$ zero eigenvalues. We introduce a coordinate transformation and adopt Assumption~\ref{assumption partially strongly convex} to guarantee exponential convergence.
Let $z = \begin{bmatrix}Q & \mathbf{0}\\ \mathbf{0} & I \end{bmatrix}z'$. Then, system \eqref{eq:error system general} becomes
\begin{align}\label{eq:error system reduced}
\dot{z}' = A' z' + B' u, \quad y = C' z', \quad u = \Delta(y)
\end{align}
where
\begin{align*}
A'
= & \begin{bmatrix}
	 	 - \alpha Q^\top  T^\top  T Q & -Q^\top  T^\top \\
	 	 T Q &  \mathbf{0} 
	 \end{bmatrix} \nonumber \\
= & \begin{bmatrix}
	 	 -\alpha R_1 R_1^\top & \mathbf{0} & - R_1\\
	 	 \mathbf{0} & \mathbf{0} & \mathbf{0}\\
	 	 R_1^\top  & \mathbf{0} & \mathbf{0}
	 \end{bmatrix} \in 
	\mathbb{R}^{(n+m) \times (n+m)},\\
B' 
  = & 
	\begin{bmatrix}
	 	- Q^\top  \\ \mathbf{0}
	 \end{bmatrix} \in \mathbb{R}^{(n+m)\times n},
 ~ C'
= 
   \begin{bmatrix}
	 	Q & \mathbf{0}
	 \end{bmatrix} \in \mathbb{R}^{n \times (n+m)}.
\end{align*}
Notice that $A'$ is not Hurwitz. Nevertheless, recall from Assumption~\ref{assumption partially strongly convex} that $g(x') = f(Qx')$ is $\mu$-partially strongly convex with respect to
the partitioned vector $x'_{S}$.
Then, \eqref{eq:error system reduced} can be rewritten as
\begin{align}\label{error system new}
& \dot{z}' = \mathcal{A} z' + B' u', ~~y' = C' z', ~~ u' = \Delta'(y')
\end{align}
where
\begin{align*}
& \mathcal{A} = 
	 \begin{bmatrix}
	 	 - \alpha R_1R_1^\top  & \mathbf{0} & -  R_1\\
	 	 \mathbf{0} & -\mu I_{m} &  \mathbf{0} \\
	 	 R_1^\top  & \mathbf{0} & \mathbf{0}
	 \end{bmatrix}\\
& \Delta'
 =	 	\nabla g(y') - \nabla g (x'^*)
	 - \mu Q \begin{bmatrix}
	 \mathbf{0}_{ (n - m) \times 1} \\ 
  \left(y' - x'^*\right)_{S}
	 \end{bmatrix}
\end{align*}
It can be similarly shown that \eqref{error system new} is passive. Since an orthogonal linear transformation does not change the Lipschitz constant \cite{zhang2015restricted}, by Assumptions~\ref{assumption general convex}, \ref{assumption differentiable lips} and \ref{assumption partially strongly convex}, $\Delta'$ also satisfies the same IQC \eqref{eq:IQC} with respect to $y'$ and $u'$. In addition, $\mathcal{A}$ is Hurwitz by Lemma~\ref{lem:Hurwitz}. Thus, following similar arguments to the previous step, exponential convergence is guaranteed.
\end{proof}
\begin{remark}
The proof above illustrates the convenience of incorporating a frequency-domain approach in the analysis, as it relies solely on a quadratic storage function rather than an explicit Lyapunov function to prove exponential stability. 
Specifically, we show that the Hurwitz property of $G(s)$ ensures that $G(j \omega) + G^{\top} (- j\omega)$ remains analytic under a small real shift. With this small perturbation, the system still preserves the essential IFP property needed for exponential stability.
However, obtaining a tighter convergence rate may require the use of a more involved IQC than \eqref{eq:IQC}, which is out of the scope of this paper \cite{lessard2016analysis,li2025convergence,li2024generalization}.
\end{remark}

The frequency-domain analysis in this work can be readily applied to distributed optimization as a special case to relax the convexity requirement to restricted secant inequality constraints, as shown in the next section.

\begin{remark}
It is worth noting that \cite{chen2020exponential} considers the problem
	\begin{align*}
		 & \min_{x \in \mathbb{R}^{n}, y \in \mathbb{R}^{m}} f(x) + g(y) \\		
		& \text{subject to~} A x + B y = d
	\end{align*}
where $A \in \mathbb{R}^{k \times n}$ is of full row rank.
This can be regarded as a special case of our method since we include a wider class of functions where the variables can be coupled.
\end{remark}

\section{Applications to Distributed Optimization}
We apply the analysis of the equality-constrained problem in the previous section to distributed optimization to relax the requirement of strong convexity.
The distributed optimization problem is
\begin{align*}
\min_{x \in \mathbb{R}^{n}} \sum_{i = 1}^{N}  f_i (x)
\end{align*}
which can be rewritten as the standard form
\begin{equation}\label{eq:distributed optimization problem}
\begin{aligned}
	& \min_{x_i \in \mathbb{R}^{n}, ~ i = 1, \ldots, N} \sum_{i = 1}^{N}  f_i (x_i)  \\
		& \qquad \text{subject to } \boldsymbol{L} x = \mathbf{0}
\end{aligned}
\end{equation}
where $x = [x_1^\top,\ldots, x_n^\top]^\top$, $\boldsymbol{L} = L \otimes I_n$, and $L$ is the Laplacian matrix for the communication graph among a group of $N$ agents.
\begin{assumption}\label{assumption graph}
	The communication graph is undirected and connected, i.e., $L = L^\top  \succeq 0$ and $0$ is its simple eigenvalue, with the associated eigenvector $\mathbf{1}_{N}$.
\end{assumption}
Under this assumption, \eqref{eq:distributed optimization problem} imposes that $x_i = x_k$, for all $i, k$, recovering the original distributed optimization problem. More on the distributed settings can be found in \cite{nagahara2016control, yang2019survey} and references therein.

\subsection{Relaxing local strong convexity}
Considering the penalty term $\frac{\alpha}{2}x^\top  \boldsymbol{L}x$ instead of $\frac{\alpha}{2}\| \boldsymbol{L} x\|^2$, the augmented primal-dual algorithm gives,
\begin{equation}\label{eq:distributed algorithm}
\begin{aligned}
	\dot{x} = & -\nabla f(x) - \alpha \boldsymbol{L} x - \boldsymbol{L} \lambda\\
	\dot{\lambda} = & \boldsymbol{L} x
\end{aligned}
\end{equation}
which is the widely studied distributed proportional-integral (PI) algorithms \cite{gharesifard2013distributed}. Though $\boldsymbol{L}$ is not a full rank matrix, we can introduce a linear transformation to eliminate its zero eigenvalues.
Define $Q = \begin{bmatrix} Q_1 & Q_2 \end{bmatrix}$ where $Q_1 \in \mathbb{R}^{ N \times (N-1)}$ and $Q_2 \in \mathbb{R}^{ N \times 1}$, and $Q_2 = \dfrac{1}{\sqrt{N}}\mathbf{1}_{N}$ is the normalized eigenvector corresponding to the zero eigenvalue of $L$.
It follows that $Q^\top  L Q = \begin{bmatrix} \Lambda & 0 \\ 0 & 0 \end{bmatrix}$, where $\Lambda \in \mathbb{R}^{(N-1) \times (N-1)}$ is a positive definite matrix.
Let $x = \mathbf{Q} x'$ with $\mathbf{Q} = Q \otimes I_n$, then, problem \eqref{eq:distributed optimization problem} becomes
\begin{align*}
&	\min_{x'_i \in \mathbb{R}^{n}, ~ i = 1, \ldots, N} \sum_{i = 1}^{N}  f_i ( \sum_{j = 1}^{N} Q_{ij} x_j'),  \nonumber \\
&		\text{subject to }  
		\begin{bmatrix}
		\mathbf{\Lambda}  & \mathbf{0}_{(N-1)n \times n}
		\end{bmatrix} x'  =: \mathbf{T} x' = \mathbf{0}
\end{align*}
where $\mathbf{\Lambda} = \Lambda \otimes I_n$
and $\mathbf{T}  \in \mathbb{R}^{(N-1)n \times Nn}$ is of full row rank since $\Lambda$ is positive definite.
Then, the analysis in the previous section can be applied to this problem.
Recall that Assumption~\ref{assumption partially strongly convex} requires $f(x)$ to be partially strongly convex only in the subspace that satisfies the equality constraint, i.e., when consensus is achieved among agents. 
Therefore, Assumption~\ref{assumption partially strongly convex} is equivalent to saying that the sum $f (x) = \sum_{i = 1}^{N} f_i(x)$ is strongly convex, which is a weaker condition than requiring strong convexity on local objective functions \cite{li2020input}. 
\begin{corollary}
Under Assumptions~\ref{assumption general convex}, \ref{assumption differentiable lips}, \ref{assumption partially strongly convex}, \ref{assumption feasible problem}, and \ref{assumption graph}, $f (x) = \sum_{i = 1}^{N} f_i(x)$ is strongly convex and the variable $x(t)$ in the distributed algorithm \eqref{eq:distributed algorithm} exponentially converges to the optimal solution of problem \eqref{eq:distributed optimization problem}.
\end{corollary}

\subsection{Relaxing global strong convexity}
In addition to relaxing local strong convexity, we notice the global strong convexity requirement can also be relaxed to the restricted secant inequality (RSI) condition \cite{zhang2015restricted}, i.e., for any $x \in \mathbb{R}^{n}$,
\begin{align}\label{eq:RSI}
\left( \nabla f(x) - \nabla f(x_p) \right)^\top  (x - x_p) \geq \mu \|x - x_p \|^2
\end{align}
where $x_p$ is the projection of $x$ into the solution set $\mathcal{X}^*$ of problem \eqref{eq:distributed optimization problem}, and $\nabla f(x_p) = \mathbf{0}$. This condition is also considered in \cite{yi2020exponential}.
Here, we assume the following.
\begin{assumption}\label{assumption RSI}
	The global objective function $\sum_{i} f_i (x)$ satisfies the restricted secant inequality (RSI) condition \eqref{eq:RSI} with $\mu > 0$ and has a unique finite minimizer.
\end{assumption}
Note that this assumption does not even imply convexity \cite{zhang2015restricted}.
By Assumption~\ref{assumption RSI} and Assumption~\ref{assumption differentiable lips}, we directly have
\begin{align}
 \left( \nabla f(x) - \nabla f(x_p) \right)^\top (x - x_p) 
 \geq \mu \|x - x_p \|^2 \nonumber\\
  \geq \frac{\mu}{l^2} \| \nabla f(x) - \nabla f(x_p) \|^2.
\end{align}
The IQC \eqref{eq:IQC} holds with $l$ changed to $\frac{l^2}{\mu}$. We provide in the following a tighter bound on the co-coercivity parameter.
\begin{lemma}
	Under Assumption~\ref{assumption RSI}, $f(x)$ satisfies the co-coercivity
\begin{align}\label{eq:rsi co-coercivity}
	\left( \nabla f(x) - \nabla f(x_p) \right)^\top  (x - x_p)
	 \geq \frac{1}{l} \| \nabla f(x) - \nabla f(x_p) \|^2.
\end{align}
\end{lemma}
\begin{proof}
Define $\phi(x) = f(x) - \nabla f(x_0)^\top  x$, for a fixed $x_0 \in \mathbb{R}^n$. It is clear that $\phi(x)$ has $l$-Lipschitz gradient. By Assumption~\ref{assumption RSI}, $\phi (x)$ has a global optimal point $x = x_0$. Then, following similar arguments from \cite[Theorem~2.1.5]{nesterov2003introductory}, we conclude \eqref{eq:rsi co-coercivity}.
\end{proof}
In addition, as $\nabla f(x^*) = \mathbf{0}$, we obtain $x^* = x_p$ by Assumption~\ref{assumption RSI}. Therefore, exponential convergence to the optimal solution is guaranteed.
\begin{corollary}
Under Assumptions~\ref{assumption differentiable lips}, \ref{assumption graph} and \ref{assumption RSI}, the variable $x(t)$ in the distributed algorithm \eqref{eq:distributed algorithm} exponentially converges to the optimal solution of problem \eqref{eq:distributed optimization problem}.
\end{corollary}
The proof is omitted as the analysis in the proof of Theorem~\ref{thm exponential convergence} remains intact.

This generalization broadens the applicability of exponential convergence results to problems that do not satisfy convex conditions. Our result coincides with \cite{yi2020exponential} and is equivalent to the EB condition in \cite{liang2019exponential} under the convexity assumption \cite{karimi2016linear}.

\section{Conclusion}
This paper established the global exponential convergence of augmented primal-dual gradient algorithms for optimization problems with partially strongly convex functions. The globally strong convexity of the objective function can be relaxed to strong convexity only in the subspace defined by the equality constraints, provided that the global Lipschitz condition for the gradient is satisfied. This relaxation allows the algorithm to achieve exponential convergence even when the objective function is merely convex outside this subspace.
Future work could further extend these results to optimization problems with inequality constraints.


\bibliographystyle{IEEEtran}
\bibliography{References}
\end{document}